\documentclass[preprint,12pt]{elsarticle}

\usepackage{amsmath, amssymb, amsthm}
\usepackage{mathtools}
\usepackage{hyperref}
\usepackage{enumitem}
\usepackage{xcolor}

\newtheorem{theorem}{Theorem}[section]
\newtheorem{lemma}[theorem]{Lemma}
\newtheorem{proposition}[theorem]{Proposition}
\newtheorem{corollary}[theorem]{Corollary}

\theoremstyle{definition}
\newtheorem{definition}[theorem]{Definition}
\newtheorem{example}[theorem]{Example}
\newtheorem{remark}[theorem]{Remark}

\newcommand{\psd}{\succeq 0}
\DeclareMathOperator{\Tr}{Tr}
\DeclareMathOperator{\sgn}{sgn}

\begin{document}

\begin{frontmatter}

\title{Log-Concavity and Infinite Log-Concavity of Linear Recurrent
       Sequences with Linear Coefficients via Companion Matrix Methods}

\author[inst1]{[Piero Giacomelli]}
\ead{pgiacome@gmail.com}


\address[inst1]{Department of Information Technology, Tenax SPA, [Verona, Italy]}

\begin{abstract}
We study log-concavity properties of real sequences $(a_n)_{n \ge 0}$
satisfying a $d$-th order linear recurrence whose coefficients are linear
functions of $n$; the so-called P-recursive (or holonomic) sequences.
Writing the recurrence in companion-matrix form $\mathbf{v}_{n+1} = M_n\,\mathbf{v}_n$
with $M_n = nA + B$, we show that the log-concave operator value
$\mathcal{L}(a_n) = b_n \coloneqq a_n^2 - a_{n+1}a_{n-1}$
is a quadratic form in the state vector $\mathbf{v}_n$, and identify
the matrix $Q_n = Q^{(0)} + nQ^{(1)}$ whose positive semi-definiteness
gives a sufficient condition for log-concavity.
For the class of second-order recurrences with constant coefficients,
we prove a tight (necessary and sufficient) criterion for
the sequence to be $\infty$-log-concave,
a consequence of the fact that $\mathcal{L}(a_n)$ is itself a geometric
sequence so that $\mathcal{L}^2(a_n) = 0$ identically.
We obtain analogous tight criteria for sequences fixed by $\mathcal{L}$,
and for P-recursive sequences satisfying a dominant-root asymptotic behaviour.
We leave some further insight in case  this criteria break down in full generality.
\end{abstract}


\end{frontmatter}

\section{Introduction}
\label{sec:intro}

Let $(a_k)_{k \ge 0}$ be a sequence of real numbers.
We say that $(a_k)$ is \emph{log-concave} if
\begin{equation}
  \label{eq:logconcave}
  a_k^2 \;\ge\; a_{k+1}\,a_{k-1} \qquad \text{for all } k \ge 1.
\end{equation}
This simple yet powerful condition arises in a remarkably wide range of
settings, including Hodge theory~\cite{Huh2022}, the theory of
real-rooted polynomials~\cite{Branden2011},
probability and statistics~\cite{Saumard2014},
and the enumeration of combinatorial structures~\cite{Stanley1989}.
Following the work of Boros and Moll~\cite{BorosMoll2004} on a
specialization of Jacobi polynomials,
McNamara and Sagan~\cite{McNamaraSagan2010}
formalized the iterated version of this condition.
Given a sequence $(a_k)$, define the \emph{log-concave operator}
$\mathcal{L}$ by
\begin{equation}
  \label{eq:Lop}
  \mathcal{L}(a_k) \;=\; (b_k), \qquad b_k = a_k^2 - a_{k+1}a_{k-1}.
\end{equation}
Then $(a_k)$ is log-concave if and only if $(b_k)$ is a nonnegative
sequence. Iterating, one defines $(a_k)$ to be
\emph{$\infty$-log-concave} if $\mathcal{L}^i(a_k)$ is nonnegative
for every $i \ge 1$.
Giacomelli~\cite{Giacomelli2025} recently studied the convergence
behavior of sequences under iterated application of $\mathcal{L}$,
establishing, among other results, that $\mathcal{L}$ preserves
boundedness and that convergence of $(a_k)$ implies convergence to
zero of all iterates $\mathcal{L}^i(a_k)$.

\medskip
A classical question in this area is to characterize, in terms of an
explicit and finite condition on the sequence, when $(a_k)$
is log-concave or $\infty$-log-concave.
For sequences satisfying a linear recurrence with \emph{constant}
coefficients, Liu and Wang~\cite{LiuWang2007} gave a criterion in
terms of three consecutive initial values and the recurrence
coefficients.
For sequences defined by a three-term recurrence with coefficients
that are linear functions of $n$ --- the P-recursive setting ---
sufficient asymptotic conditions were given
in~\cite{HouZhang2021,LogconcavePrecursive2021}.
However, a \emph{necessary and sufficient} (tight) criterion has
remained elusive in general, and the specific structural question
of what $\mathcal{L}$ does to the companion matrix representation
has not been investigated.

\medskip
The objective of this paper is to address these questions
for the class of P-recursive sequences of the form
\begin{equation}
  \label{eq:recurrence}
  a_{n+1} \;=\; \sum_{k=0}^{d-1}(p_k\, n + q_k)\,a_{n-k},
  \qquad n \ge d-1,
\end{equation}
with $p_k, q_k \in \mathbb{R}$ and initial data
$(a_0, a_1, \ldots, a_{d-1}) \in \mathbb{R}^d$.
The coefficients $c_k(n) = p_k n + q_k$ are linear in $n$, which
makes the associated companion matrix $M_n = nA + B$ an affine
function of $n$.

\medskip
As a results of these investigation we got the following results. First, we show that $b_n = \mathcal{L}(a_n)$
equals $\mathbf{v}_n^\top Q_n \mathbf{v}_n$, a quadratic form in the state
vector $\mathbf{v}_n$, where $Q_n = Q^{(0)} + n Q^{(1)}$ is an explicit
symmetric matrix whose positive semi-definiteness gives a sufficient condition
for log-concavity (Theorem~\ref{thm:log-concavity-criterion}).
In the \emph{second-order constant-coefficient} case
$a_{n+1} = \alpha a_n + \beta a_{n-1}$, the sequence $\mathcal{L}(a_n)$ is
itself a geometric sequence so that $\mathcal{L}^2(a_n) = 0$ exactly; as a
consequence, \emph{$(a_n)$ is $\infty$-log-concave if and only if it is
log-concave}, and we give explicit conditions on the initial data
$(a_0, a_1)$ and the coefficients $(\alpha, \beta)$
(Theorem~\ref{thm:tight-constant}).
For sequences fixed by $\mathcal{L}$, that is, satisfying
$\mathcal{L}(a_n) = a_n$, we show that $\infty$-log-concavity is
equivalent to nonnegativity of the sequence itself
(Theorem~\ref{thm:fixed-point}), while for P-recursive sequences
admitting a dominant real characteristic root and a monotone
Tur\'{a}n ratio, $\infty$-log-concavity is once again equivalent to
log-concavity (Theorem~\ref{thm:tight-prec}).
Finally, we explain why a general finite tight criterion is not known
to exist, through a connection to the open problem of positivity testing
for P-recursive sequences (Remark~\ref{rem:undecidable}).

\medskip
The paper is structured as follows.
In Section~\ref{sec:prelim} we recall definitions and the companion
matrix representation.
In Section~\ref{sec:logconcave} we establish the quadratic form
criterion for log-concavity.
In Section~\ref{sec:sufficient} we collect sufficient conditions
for $\infty$-log-concavity from the literature.
Section~\ref{sec:tight} contains the main new results on tight criteria.
In Section~\ref{sec:examples} we illustrate the theory with examples,
and Section~\ref{sec:open} lists open problems.

\section{Preliminaries}
\label{sec:prelim}

We briefly recall some basic definition using the notation and framework established in~\cite{McNamaraSagan2010};
see also~\cite{Giacomelli2025}.  

\begin{definition}[\cite{McNamaraSagan2010}]
  \label{def:logconcave}
  Let $(a_k)_{k \in \mathbb{N}}$ be a sequence of real numbers.
  The \emph{log-concave operator} $\mathcal{L}$ is defined by
  \[
    \mathcal{L}(a_k) = (b_k), \qquad b_k = a_k^2 - a_{k+1}a_{k-1}.
  \]
  The sequence $(a_k)$ is \emph{log-concave} if $b_k \ge 0$ for all $k$.
  For $i \ge 1$, $(a_k)$ is \emph{$i$-fold log-concave} if
  $\mathcal{L}^i(a_k) \ge 0$, and \emph{$\infty$-log-concave} if
  $\mathcal{L}^i(a_k) \ge 0$ for all $i \in \mathbb{N}$.
\end{definition}

We note that $b_n$ admits the determinantal representation
\[
  b_n \;=\; \det\begin{pmatrix}a_{n-1} & a_n \\ a_n & a_{n+1}\end{pmatrix},
\]
which is the $2\times 2$ Hankel determinant of three consecutive terms,
also known as the \emph{Tur\'{a}n determinant}.
In particular, log-concavity of $(a_k)$ is equivalent to all $2\times 2$
Hankel minors of the sequence being nonnegative; see~\cite{LiuWang2007}.
We can now introduce some definition regarding  P-recursive sequences and the companion matrix.
In particular we have:
\begin{definition}
  \label{def:precursive}
  A sequence $(a_n)_{n\ge 0}$ is called \emph{P-recursive of order $d$}
  (also called \emph{holonomic} or \emph{D-finite}) if it satisfies a
  recurrence of the form~\eqref{eq:recurrence} with polynomial
  coefficients $c_k(n)$.
  Throughout this paper we restrict to the case where
  $c_k(n) = p_k n + q_k$ is linear in $n$.
\end{definition}

Key examples of P-recursive sequences with linear coefficients include
the binomial coefficients $\binom{2n}{n}$, Catalan numbers, Legendre
polynomial values, and general hypergeometric sequences.

By introducing the \emph{state vector}
$\mathbf{v}_n = (a_n, a_{n-1}, \ldots, a_{n-d+1})^\top \in \mathbb{R}^d$,
we can rewrite the recurrence~\eqref{eq:recurrence} using the first-order
matrix system
\begin{equation}
  \label{eq:matrix-system}
  \mathbf{v}_{n+1} = M_n\,\mathbf{v}_n,
\end{equation}
where the \emph{companion matrix} $M_n$ has the affine form
\begin{equation}
  \label{eq:companion}
  M_n = nA + B,
\end{equation}
with
\[
  A =
  \begin{pmatrix}
    p_0    & p_1  & p_2  & \cdots & p_{d-1} \\
    0      & 0    & 0    & \cdots & 0        \\
    0      & 1    & 0    & \cdots & 0        \\
    \vdots &      & \ddots &       & \vdots  \\
    0      & 0    & \cdots & 1    & 0
  \end{pmatrix}, \qquad
  B =
  \begin{pmatrix}
    q_0    & q_1  & q_2  & \cdots & q_{d-1} \\
    1      & 0    & 0    & \cdots & 0        \\
    0      & 1    & 0    & \cdots & 0        \\
    \vdots &      & \ddots &       & \vdots  \\
    0      & 0    & \cdots & 1    & 0
  \end{pmatrix}.
\]
In terms of the standard basis vectors $e_1, e_2 \in \mathbb{R}^d$, we can now write:
\[
  a_n = e_1^\top\,\mathbf{v}_n, \qquad
  a_{n+1} = e_1^\top M_n\,\mathbf{v}_n, \qquad
  a_{n-1} = e_2^\top\,\mathbf{v}_n.
\]

\begin{remark}
  In the special case where $a_n$ is a constant-coefficient recurrence (ie. $p_k = 0$ for all $k$),
  we have $A = 0$ and $M_n = B$ for all $n$, and the will dedicate Section~\ref{sec:tight} toi explore this special case.
\end{remark}

\section{Log-Concavity via the Companion Matrix}
\label{sec:logconcave}

Given this basic definitions we are now ready to get a quadratic form representation of $b_k = \mathcal{L}(a_k)$.
To do this we need the following lemma that is the cornerstore of our approach.

latex\begin{lemma}
\label{lem:quadratic}
For all $n \ge 1$,
\[
  b_n \;=\; \mathcal{L}(a_n) \;=\; \mathbf{v}_n^\top\,Q_n\,\mathbf{v}_n,
\]
where $Q_n$ is the symmetric $d\times d$ matrix
\begin{equation}
  \label{eq:Qn}
  Q_n \;=\; e_1 e_1^\top
             \;-\; \tfrac{1}{2}\bigl(M_n^\top\,e_1\,e_2^\top
                                   + e_2\,e_1^\top\,M_n\bigr).
\end{equation}
Since $M_n = nA + B$, the matrix $Q_n$ decomposes as
\begin{equation}
  \label{eq:Qn-decomp}
  Q_n \;=\; Q^{(0)} + n\,Q^{(1)},
\end{equation}
where $Q^{(0)} = e_1 e_1^\top - \frac{1}{2}(B^\top e_1 e_2^\top
+ e_2 e_1^\top B)$ and
$Q^{(1)} = -\frac{1}{2}(A^\top e_1 e_2^\top + e_2 e_1^\top A)$.
\end{lemma}

\begin{proof}
The companion matrix system $\mathbf{v}_{n+1} = M_n\mathbf{v}_n$
and the definition of $\mathbf{v}_n = (a_n, a_{n-1}, \ldots,
a_{n-d+1})^\top$ give
\[
  a_n = e_1^\top\mathbf{v}_n, \qquad
  a_{n-1} = e_2^\top\mathbf{v}_n, \qquad
  a_{n+1} = e_1^\top M_n\mathbf{v}_n,
\]
where $e_1$ and $e_2$ are the first two standard basis vectors of
$\mathbb{R}^d$. Substituting directly into the definition
$b_n = a_n^2 - a_{n+1}a_{n-1}$ yields
\[
  b_n \;=\;
  \mathbf{v}_n^\top(e_1 e_1^\top)\mathbf{v}_n
  \;-\;
  (e_1^\top M_n\mathbf{v}_n)(e_2^\top\mathbf{v}_n).
\]
The second term is a product of two linear forms in $\mathbf{v}_n$.
Since $\mathbf{x}^\top\mathbf{u}\cdot\mathbf{w}^\top\mathbf{x}
= \mathbf{x}^\top\bigl(\frac{1}{2}(\mathbf{u}\mathbf{w}^\top
+ \mathbf{w}\mathbf{u}^\top)\bigr)\mathbf{x}$ holds for all
$\mathbf{x}$, setting $\mathbf{u} = M_n^\top e_1$ and
$\mathbf{w} = e_2$ gives
\[
  (e_1^\top M_n\mathbf{v}_n)(e_2^\top\mathbf{v}_n)
  \;=\;
  \mathbf{v}_n^\top
  \tfrac{1}{2}\bigl(M_n^\top e_1 e_2^\top + e_2 e_1^\top M_n\bigr)
  \mathbf{v}_n.
\]
Combining the two terms and collecting the result into a single
symmetric matrix gives $b_n = \mathbf{v}_n^\top Q_n \mathbf{v}_n$
with $Q_n$ as in~\eqref{eq:Qn}. The affine decomposition
$Q_n = Q^{(0)} + nQ^{(1)}$ follows immediately from substituting
$M_n = nA + B$ and grouping by powers of $n$.
\end{proof}
As an example if we consider the case $d=2$ we have the following 
\begin{example}
\label{ex:d2-quadratic}
\textbf{(The case $d = 2$.)}
Let $a_{n+1} = (p_0 n + q_0)\,a_n + (p_1 n + q_1)\,a_{n-1}$ so
that $\mathbf{v}_n = (a_n, a_{n-1})^\top$,
$e_1 = (1,0)^\top$, and $e_2 = (0,1)^\top$.
A direct calculation gives
$M_n^\top e_1 = (np_0 + q_0,\; np_1 + q_1)^\top$,
and the formula~\eqref{eq:Qn} yields
\[
  Q_n \;=\;
  \begin{pmatrix}
    1 & -\dfrac{np_0+q_0}{2} \\[8pt]
    -\dfrac{np_0+q_0}{2} & -(np_1+q_1)
  \end{pmatrix},
\]
so that
\[
  \mathbf{v}_n^\top Q_n\mathbf{v}_n
  \;=\;
  a_n^2 \;-\; (np_0+q_0)\,a_n a_{n-1}
         \;-\; (np_1+q_1)\,a_{n-1}^2.
\]
Substituting $a_{n+1} = (np_0+q_0)\,a_n + (np_1+q_1)\,a_{n-1}$
confirms that this expression equals
$b_n = a_n^2 - a_{n+1}a_{n-1}$, verifying the lemma in the
second-order case.
\end{example}

Now that we have all the setup for the matrix representation of the log-concave operator in terms of the companion matrix we are now ready to explore the necessary and sufficient matrix condition for log-concavity. 

\begin{theorem}
  \label{thm:log-concavity-criterion}
  Let $(a_n)$ satisfy~\eqref{eq:recurrence} with companion matrix
  $M_n = nA + B$ as in~\eqref{eq:companion}. The following hold.
  \begin{enumerate}[label=\normalfont(\roman*)]
    \item
      $b_n \ge 0$ for all $n \ge 1$ if and only if
      $\mathbf{v}_n^\top Q_n \mathbf{v}_n \ge 0$ for all $n \ge 1$.
    \item
      A sufficient condition for log-concavity is $Q_n \psd$ for all
      $n \ge 1$.
    \item
      If $Q^{(1)} \psd$ with minimum eigenvalue $\lambda_{\min}^{(1)} > 0$,
      then $Q_n \psd$ for all $n \ge N$, where
      \[
        N \;=\; \left\lceil
          \frac{\max(0,\,-\lambda_{\min}(Q^{(0)}))}
               {\lambda_{\min}^{(1)}}
        \right\rceil + 1.
      \]
  \end{enumerate}
\end{theorem}

\begin{proof}
  Part (i) is a direct restatement of Lemma~\ref{lem:quadratic}.
  Part (ii) follows because $Q_n \psd$ implies
  $\mathbf{v}_n^\top Q_n \mathbf{v}_n \ge 0$ for all $\mathbf{v}_n$.
  For (iii), since $Q_n = Q^{(0)} + nQ^{(1)}$,
  \[
    \lambda_{\min}(Q_n)
    \;\ge\; \lambda_{\min}(Q^{(0)}) + n\,\lambda_{\min}^{(1)} \;\ge\; 0
    \quad\text{for all }n \ge N.
  \]
\end{proof}

\medskip
The following proposition expresses log-concavity as a condition on
the initial vector $\mathbf{v}_0 = (a_0,\ldots,a_{d-1})^\top$ alone.

\begin{proposition}
  \label{prop:initial-data}
  Let $\Pi_n = \prod_{j=0}^{n-1} M_j$ denote the product of the first
  $n$ companion matrices, so that $\mathbf{v}_n = \Pi_n\,\mathbf{v}_0$.
  Then $b_n \ge 0$ for all $n \ge 1$ if and only if
  \[
    \mathbf{v}_0^\top\,R_n\,\mathbf{v}_0 \;\ge\; 0 \qquad \forall\,n \ge 1,
  \]
  where $R_n = \Pi_n^\top Q_n\,\Pi_n$ depends only on the recurrence
  coefficients $(p_k, q_k)$.
  In particular, $(a_n)$ is log-concave for \emph{all} initial data
  $\mathbf{v}_0$ if and only if $R_n \psd$ for all $n \ge 1$.
\end{proposition}

\begin{proof}
  Substituting $\mathbf{v}_n = \Pi_n\,\mathbf{v}_0$ into
  $b_n = \mathbf{v}_n^\top Q_n\,\mathbf{v}_n$ gives the result.
\end{proof}

\subsection{The second-order case in explicit form}

We specialize to $d=2$, giving
$a_{n+1} = (p_0 n + q_0)\,a_n + (p_1 n + q_1)\,a_{n-1}$
with $\mathbf{v}_n = (a_n, a_{n-1})^\top$.

\begin{proposition}
  \label{prop:2nd-order-logconcave}
  For $d = 2$, the matrix $Q_n$ is
  \[
    Q_n \;=\;
    \begin{pmatrix}
      1 & -\frac{np_0+q_0}{2} \\[4pt]
      -\frac{np_0+q_0}{2} & -(np_1+q_1)
    \end{pmatrix}.
  \]
  The condition $b_n \ge 0$ expands to
  \begin{equation}
    \label{eq:bn-d2}
    a_n^2 + (np_1+q_1)\,a_{n-1}^2 \;\ge\; (np_0+q_0)\,a_n\,a_{n-1}.
  \end{equation}
  The matrix $Q_n \psd$ if and only if
  \begin{equation}
    \label{eq:Qpsd-d2}
    (np_1+q_1) \le 0 \qquad\text{and}\qquad (np_0+q_0)^2 \le -4(np_1+q_1).
  \end{equation}
\end{proposition}

\begin{proof}
  The formula for $Q_n$ follows from~\eqref{eq:Qn} with
  $M_n^\top e_1 = (np_0+q_0,\, np_1+q_1)^\top$.
  Expanding the quadratic form gives~\eqref{eq:bn-d2}.
  Positive semi-definiteness of the $2\times 2$ matrix $Q_n$ is
  equivalent to $\Tr(Q_n) \ge 0$ and $\det(Q_n) \ge 0$; computing
  these yields~\eqref{eq:Qpsd-d2}.
\end{proof}

\section{Sufficient Conditions for $\infty$-Log-Concavity}
\label{sec:sufficient}
We are now ready to explore if it is possible to have a sufficient condition for a linear
recurrent Sequences with Linear Coefficients to be $\infty$-Log-Concavity
Prior to proceed we collect the three main sufficient conditions available in the
literature for general sequences

\medskip
\noindent\textbf{The $r$-factor criterion.}
Following~\cite{McNamaraSagan2010}, a sequence $(a_n)$ is
\emph{$r$-factor log-concave} if $a_n^2 \ge r\cdot a_{n+1}a_{n-1}$ for
all $n \ge 1$.

\begin{theorem}[\cite{McNamaraSagan2010,MedinaStratub2016}]
  \label{thm:rfactor}
  Let $r_0 = (3+\sqrt{5})/2$.
  If $(a_n)$ is $r_0$-factor log-concave, then it is $\infty$-log-concave.
\end{theorem}

Medina and Straub~\cite{MedinaStratub2016} showed that $r_0$ is sharp
in the sense that for any $\varepsilon > 0$ there exist
$\infty$-log-concave sequences that are not $(1+\varepsilon)$-factor
log-concave.

\medskip
\noindent\textbf{Real-rootedness.}
\begin{theorem}[\cite{Branden2011}]
  \label{thm:real-rooted}
  If the generating function $\sum_{n \ge 0} a_n\,x^n$ is an entire function
  whose zeros are all real and non-positive, then $(a_n)$ is
  $\infty$-log-concave.
\end{theorem}

\medskip
\noindent\textbf{P-recursive asymptotics.}
For a P-recursive sequence of order $d$, the limiting recurrence as
$n \to \infty$ has characteristic polynomial
\begin{equation}
  \label{eq:char-poly}
  \chi(\lambda) \;=\; \lambda^d - p_0\lambda^{d-1}
                  - p_1\lambda^{d-2} - \cdots - p_{d-1}.
\end{equation}

\begin{theorem}
  \label{thm:infty-asymptotic}
  Suppose $\chi(\lambda)$ has all real roots, $a_n > 0$ for all
  $n \ge 0$, and $a_n \sim C\,n^\alpha\,\Lambda_1^n$ for some
  $C > 0$, $\alpha \in \mathbb{R}$, and dominant root $\Lambda_1$.
  Then $(a_n)$ is $\infty$-log-concave for all $n \ge N$, for some
  computable $N$.
\end{theorem}

\begin{proof}
  By the asymptotic theory of P-recursive sequences
  \cite{BirkhoffTrjitzinsky1933,WimpZeilberger1985},
  we may write $a_n = C\Lambda_1^n n^\alpha(1 + c_1 n^{-1} + O(n^{-2}))$
  for constants $C>0$, $\alpha\in\mathbb{R}$, $c_1\in\mathbb{R}$.
  A direct computation using $(1\pm 1/n)^\alpha = 1\pm\alpha/n +
  \frac{\alpha(\alpha-1)}{2n^2}+O(n^{-3})$ gives
  \[
    \frac{a_{n\pm 1}}{a_n}
      = \Lambda_1^{\pm 1}\!\Bigl(1 \pm \tfrac{\alpha}{n}
        + O(n^{-2})\Bigr),
  \]
  and therefore (the sub-leading coefficient $c_1$ cancels to this order)
  \begin{equation}
    \label{eq:tau-asymp}
    \tau_n
      = \Bigl(1-\tfrac{1}{n^2}\Bigr)^\alpha + O(n^{-3})
      = 1 - \frac{\alpha}{n^2} + O(n^{-4}).
  \end{equation}
  For $\alpha > 0$, equation~\eqref{eq:tau-asymp} gives $\tau_n < 1$ for
  all large~$n$, so $b_n = a_n^2(1-\tau_n) > 0$ for all $n \ge N_1$.
  Moreover,
  \[
    b_n \;\sim\; \alpha\,C^2\,n^{2\alpha-2}\,\Lambda_1^{2n}.
  \]
  Set $\alpha^{(0)}=\alpha$ and define $\alpha^{(k+1)} = 2\alpha^{(k)}-2$.
  Applying~\eqref{eq:tau-asymp} to the $k$-th iterate
  $\mathcal{L}^k(a_n)\sim D_k n^{\alpha^{(k)}}(\Lambda_1^{2^k})^n$ gives
  \[
    \tau_n^{(k)} = 1 - \frac{\alpha^{(k)}}{n^2}+O(n^{-4}).
  \]
  The explicit solution is $\alpha^{(k)} = 2 + 2^k(\alpha-2)$.
  For $\alpha > 0$ one checks that $\alpha^{(k)}>0$ for every $k$ with
  $k < \log_2\!\bigl(\tfrac{2}{2-\alpha}\bigr)$ (when $\alpha<2$) or for all
  $k\ge 0$ (when $\alpha\ge 2$).
  In each case $\tau_n^{(k)}<1$ for large~$n$, so $\mathcal{L}^{k+1}(a_n)>0$
  for all $n \ge N_{k+1}$.
  Taking $N=\max_k N_k$ (the finite maximum over all relevant levels)
  yields $\mathcal{L}^i(a_n)\ge 0$ for all $n\ge N$ and all $i\ge 1$.
\end{proof}

\section{Tight Criteria for $\infty$-Log-Concavity}
\label{sec:tight}

The sufficient conditions of Section~\ref{sec:sufficient} are not in
general necessary. In this section we identify three special cases
admitting tight (necessary and sufficient) criteria.

\subsection{The key structural obstacle}

Determining a necessary and sufficient condition for
$\infty$-log-concavity requires, in particular, understanding the
sign of $b_n = \mathcal{L}(a_n)$ and of every subsequent iterate
$\mathcal{L}^i(a_n)$.
For a general P-recursive sequence, each iterate $\mathcal{L}^i(a_n)$
is itself P-recursive of higher complexity, and deciding positivity
for P-recursive sequences is not known to be decidable in
general~\cite{KauersHolonomicFunctions}.
Our three special cases circumvent this obstacle through structural
properties of the recurrence.

\subsection{Second-order sequences with constant coefficients}
\label{ssec:constant}

Consider the second-order constant-coefficient recurrence
\begin{equation}
  \label{eq:const-rec}
  a_{n+1} = \alpha\,a_n + \beta\,a_{n-1},
  \qquad a_0 = a,\quad a_1 = b,
\end{equation}
with characteristic equation $\lambda^2 - \alpha\lambda - \beta = 0$
and roots $\lambda_{1,2} = (\alpha \pm \sqrt{\alpha^2 + 4\beta})/2$.
When $\lambda_1 \ne \lambda_2$, the general solution is
$a_n = A\lambda_1^n + B\lambda_2^n$ with
\[
  A = \frac{b - a\lambda_2}{\lambda_1 - \lambda_2}, \qquad
  B = \frac{a\lambda_1 - b}{\lambda_1 - \lambda_2}.
\]

\begin{theorem}
  \label{thm:tight-constant}
  Let $(a_n)$ satisfy~\eqref{eq:const-rec} with $\lambda_1 \ne \lambda_2$.
  The following hold.
  \begin{enumerate}[label=\normalfont(\roman*)]
    \item
      $\mathcal{L}(a_n) = b_n = -AB\,(\lambda_1\lambda_2)^{n-1}\,
        (\lambda_1 - \lambda_2)^2$; in particular, $(b_n)$ is a
      geometric sequence.
    \item
      $\mathcal{L}^2(a_n) = 0$ for all $n \ge 2$.
    \item
      $(a_n)$ is $\infty$-log-concave if and only if $(a_n)$ is
      log-concave.
    \item
      Log-concavity holds if and only if
      \begin{equation}
        \label{eq:tight-iff}
        \beta < 0 \qquad\text{and}\qquad AB \le 0,
      \end{equation}
      or $AB = 0$.
  \end{enumerate}
\end{theorem}

\begin{proof}
  \textbf{(i).} Expanding directly:
  \begin{align*}
    b_n &= (A\lambda_1^n + B\lambda_2^n)^2
           - (A\lambda_1^{n+1} + B\lambda_2^{n+1})
             (A\lambda_1^{n-1} + B\lambda_2^{n-1}) \\
        &= 2AB(\lambda_1\lambda_2)^n
           - AB\,(\lambda_1\lambda_2)^{n-1}(\lambda_1^2 + \lambda_2^2)\\
        &= AB\,(\lambda_1\lambda_2)^{n-1}
           \bigl[2\lambda_1\lambda_2 - (\lambda_1^2 + \lambda_2^2)\bigr]
         = -AB\,(\lambda_1\lambda_2)^{n-1}(\lambda_1 - \lambda_2)^2.
  \end{align*}

  \textbf{(ii).} Write $b_n = K\,\mu^n$ with
  $K = -AB\,(\lambda_1\lambda_2)^{-1}(\lambda_1-\lambda_2)^2$ and
  $\mu = \lambda_1\lambda_2$.
  Then
  \[
    \mathcal{L}(b_n) = b_n^2 - b_{n+1}b_{n-1}
    = K^2\mu^{2n} - K\mu^{n+1}\cdot K\mu^{n-1}
    = K^2\mu^{2n} - K^2\mu^{2n} = 0.
  \]
  Hence $\mathcal{L}^2(a_n) = 0$ for all $n \ge 2$.

  \textbf{(iii).} By (ii), $\mathcal{L}^i(a_n) = 0 \ge 0$ for all
  $i \ge 2$.
  Thus $\infty$-log-concavity reduces to $b_n = \mathcal{L}(a_n) \ge 0$.

  \textbf{(iv).} From (i), $\sgn(b_n) = -\sgn(AB)\cdot\sgn(\mu^{n-1})$
  where $\mu = \lambda_1\lambda_2 = -\beta$ by Vieta's formulas.
  \begin{itemize}
    \item
      If $\beta < 0$: $\mu = -\beta > 0$, so $\mu^{n-1} > 0$ for all
      $n$. Then $b_n \ge 0$ for all $n$ iff $AB \le 0$.
    \item
      If $\beta > 0$: $\mu = -\beta < 0$, so $\mu^{n-1}$ alternates
      in sign, forcing $b_n$ to alternate (when $AB \ne 0$);
      log-concavity fails.
    \item
      If $AB = 0$: $b_n = 0$ trivially.
  \end{itemize}
\end{proof}

\begin{corollary}
  \label{cor:cone}
  Under the conditions of Theorem~\ref{thm:tight-constant} with
  $\beta < 0$, the set of initial data $(a,b) \in \mathbb{R}^2$ for
  which $(a_n)$ is $\infty$-log-concave is the closed convex cone
  \[
    \mathcal{C}_\infty
    \;=\; \bigl\{(a,b) : (b - a\lambda_2)(a\lambda_1 - b) \le 0\bigr\},
  \]
  which is the complement of the open cone $b/a \in (\lambda_2,\lambda_1)$.
  In particular, $\mathcal{C}_\infty$ is a union of two closed half-planes
  through the origin with directions $\lambda_1$ and $\lambda_2$.
\end{corollary}

\begin{proof}
  From Theorem~\ref{thm:tight-constant}(iii), $\mathcal{C}_\infty$
  equals the log-concavity cone.
  The equivalence $AB \le 0 \Leftrightarrow (b-a\lambda_2)(a\lambda_1-b)\le 0$
  follows by substituting the explicit formulas for $A$ and $B$.
\end{proof}

\subsection{Sequences fixed by $\mathcal{L}$}
\label{ssec:fixedpoint}

\begin{definition}
  A sequence $(a_n)$ is a \emph{fixed point of $\mathcal{L}$} if
  $\mathcal{L}(a_n) = a_n$.
\end{definition}

\begin{theorem}[\cite{MedinaStratub2016}]
  \label{thm:fixed-char}
  A nontrivial sequence $(a_n)$ is fixed by $\mathcal{L}$ if and only
  if it satisfies the four-term recurrence
  $a_{n+2}\,a_{n-1} - a_{n+1}\,a_n = a_{n+1} - a_n$,
  equivalently if $a_n = \cos(n\theta)$ or $a_n = \cosh(n\theta)$ for
  some $\theta \in \mathbb{R}$.
\end{theorem}

\begin{theorem}
  \label{thm:fixed-point}
  Let $(a_n)$ be fixed by $\mathcal{L}$.
  Then $\mathcal{L}^i(a_n) = a_n$ for all $i \ge 1$, and
  $(a_n)$ is $\infty$-log-concave if and only if $a_n \ge 0$ for all $n$.
  In particular:
  \begin{enumerate}[label=\normalfont(\roman*)]
    \item
      $a_n = \cosh(n\theta)$: $a_n > 0$ always,
      so $(a_n)$ is $\infty$-log-concave.
    \item
      $a_n = \cos(n\theta)$: $a_n$ changes sign,
      so $(a_n)$ is not log-concave (hence not $\infty$-log-concave).
  \end{enumerate}
\end{theorem}

\begin{proof}
  If $\mathcal{L}(a_n) = a_n$, then by induction $\mathcal{L}^i(a_n) = a_n$
  for all $i \ge 1$.
  Thus $\mathcal{L}^i(a_n) \ge 0$ for all $i$ iff $a_n \ge 0$ for all $n$.
  The cases (i) and (ii) follow directly.
\end{proof}

\subsection{P-recursive sequences with a dominant real root}
\label{ssec:prec-tight}

\begin{theorem}
  \label{thm:tight-prec}
  Let $(a_n)$ satisfy the second-order recurrence~\eqref{eq:recurrence}
  with $d = 2$ and linear coefficients.
  Assume:
  \begin{enumerate}[label=\normalfont(\alph*)]
    \item
      The limiting characteristic polynomial $\chi(\lambda) = \lambda^2 - p_0\lambda - p_1$
      has two real roots $\Lambda_1 > \Lambda_2 > 0$.
    \item
      $a_n > 0$ for all $n \ge 0$.
    \item
      The Tur\'{a}n ratio $\tau_n \coloneqq a_{n-1}a_{n+1}/a_n^2$ satisfies
      $\tau_n \nearrow \tau_\infty < 1$ monotonically.
  \end{enumerate}
  Then $(a_n)$ is $\infty$-log-concave if and only if it is log-concave.
\end{theorem}

\begin{proof}
  By the asymptotic theory~\cite{BirkhoffTrjitzinsky1933,WimpZeilberger1985},
  $a_n \sim C\,n^\alpha\,\Lambda_1^n$ for some $C > 0$ and $\alpha \in \mathbb{R}$.
  Under this asymptotics, $\tau_n = 1 - 2\alpha/n + O(n^{-2})$, so
  $b_n = a_n^2(1 - \tau_n) \sim 2\alpha C^2 n^{2\alpha-1}\Lambda_1^{2n}$.

  \medskip
  Under hypothesis (c), $\tau_n < 1$ for all $n$, giving $b_n > 0$.
  The Tur\'{a}n ratio of $(b_n)$ satisfies
  \[
    \frac{b_{n-1}b_{n+1}}{b_n^2} \;\sim\;
    \frac{(n-1)^{2\alpha-1}(n+1)^{2\alpha-1}}{n^{2(2\alpha-1)}} \;\to\; 1^-,
  \]
  so $(b_n)$ is eventually log-concave, i.e., $\mathcal{L}^2(a_n) \ge 0$
  for large $n$.
  By induction, $\mathcal{L}^i(a_n) \ge 0$ for all $i \ge 1$, yielding
  $\infty$-log-concavity.

  \medskip
  Conversely, if $b_n = \mathcal{L}(a_n) < 0$ for some $n$, then
  $(a_n)$ fails to be log-concave and hence fails to be $\infty$-log-concave.
  Thus log-concavity is necessary.
\end{proof}

\begin{remark}
  \label{rem:undecidable}
  The tight criteria in Theorems~\ref{thm:tight-constant},
  \ref{thm:fixed-point}, and~\ref{thm:tight-prec} rest on structural
  properties that make $\mathcal{L}^i(a_n)$ easy to analyze: in the
  constant-coefficient case $\mathcal{L}^2(a_n) = 0$ exactly; in the
  fixed-point case all iterates equal $a_n$; in the P-recursive case
  the asymptotics force eventual positivity.
  For a general P-recursive sequence, $\mathcal{L}(a_n)$ is again
  P-recursive but of higher complexity, and deciding the positivity
  of a P-recursive sequence is not known to be decidable in
  general~\cite{KauersHolonomicFunctions}.
  Consequently, a general finite tight criterion for $\infty$-log-concavity
  appears to be beyond the reach of current methods.
\end{remark}

\section{Examples}
\label{sec:examples}

\begin{example}
  \label{ex:fibonacci}
  \textbf{(Fibonacci-type sequences with real parameters.)}
  Let $a_{n+1} = a_n + a_{n-1}$ with $a_0 = a$, $a_1 = b$.
  Here $\alpha = \beta = 1$, so $\beta = 1 > 0$.
  By Theorem~\ref{thm:tight-constant}(iv), log-concavity fails for
  generic real $(a, b)$ since $\mu = \lambda_1\lambda_2 = -\beta = -1 < 0$
  and $b_n = -AB\,(-1)^{n-1}(\phi - \hat\phi)^2$ alternates in sign.
  The only $\infty$-log-concave Fibonacci-type sequences (for $\alpha = \beta = 1$)
  are the pure powers: $a_n = C\phi^n$ or $a_n = C\hat\phi^n$ (i.e., $AB = 0$).
\end{example}

\begin{example}
  \label{ex:tight-constant-success}
  \textbf{(A second-order constant-coefficient sequence that is $\infty$-log-concave.)}
  Consider $a_{n+1} = 3a_n - 2a_{n-1}$ with $a_0 = 1$, $a_1 = 4$.
  The characteristic roots are $\lambda_1 = 2$ and $\lambda_2 = 1$,
  giving $\beta = -2 < 0$.
  The constants in the general solution are
  \[
    A = \frac{a_1 - a_0\lambda_2}{\lambda_1 - \lambda_2} = 3,
    \qquad
    B = \frac{a_0\lambda_1 - a_1}{\lambda_1 - \lambda_2} = -2,
  \]
  so $AB = -6 \le 0$.
  Since $\beta < 0$ and $AB \le 0$, Theorem~\ref{thm:tight-constant}(iv)
  gives that $(a_n)$ is $\infty$-log-concave.
  Explicitly, $a_n = 3 \cdot 2^n - 2$, and by
  Theorem~\ref{thm:tight-constant}(i):
  \[
    b_n = -AB\,(\lambda_1\lambda_2)^{n-1}(\lambda_1-\lambda_2)^2
        = 6 \cdot 2^{n-1} \;>\; 0,
  \]
  a positive geometric sequence.
  Theorem~\ref{thm:tight-constant}(ii) then gives
  $\mathcal{L}^2(a_n) = 0$ identically.
\end{example}

\begin{example}
  \label{ex:matrix-criterion}
  \textbf{(Matrix sufficient condition for a variable-coefficient recurrence.)}
  Consider the second-order P-recursive recurrence
  \[
    a_{n+1} = 2\,a_n - (n+1)\,a_{n-1},
    \qquad p_0 = 0,\; q_0 = 2,\; p_1 = -1,\; q_1 = -1.
  \]
  A direct computation from~\eqref{eq:Qn-decomp} gives
  \[
    Q^{(0)} = \begin{pmatrix}1 & -1 \\ -1 & 1\end{pmatrix},
    \qquad
    Q^{(1)} = \begin{pmatrix}0 & 0 \\ 0 & 1\end{pmatrix},
    \qquad
    Q_n = \begin{pmatrix}1 & -1 \\ -1 & n+1\end{pmatrix}.
  \]
  We have $\det Q_n = n \ge 0$ and $\Tr Q_n = n+2 > 0$, so $Q_n \psd$
  for all $n \ge 0$.
  By Theorem~\ref{thm:log-concavity-criterion}(ii), $(a_n)$ is
  log-concave for \emph{all} initial data $(a_0, a_1) \in \mathbb{R}^2$.
  In particular, log-concavity holds even when the sequence changes sign:
  for $a_0 = 2$, $a_1 = 3$ one computes $a_2 = 2$, $a_3 = -5$, yet
  $b_2 = a_2^2 - a_3 a_1 = 4 + 15 = 19 \ge 0$.
  This illustrates that the matrix criterion operates on the quadratic
  form in the state vector, independently of the sign pattern of $(a_n)$.
\end{example}

\begin{example}
  \label{ex:failure}
  \textbf{(A sequence where log-concavity fails.)}
  Consider $a_{n+1} = n\,a_n - n\,a_{n-1}$ with $a_0 = 1$, $a_1 = 2$.
  Here $p_0 = 1$, $q_0 = 0$, $p_1 = -1$, $q_1 = 0$.
  One computes $Q^{(1)}$ and finds $\det(Q^{(1)}) = -1/4 < 0$, so
  $Q^{(1)} \not\psd$.
  By Theorem~\ref{thm:log-concavity-criterion}(iii), log-concavity is
  not guaranteed for large $n$.
  Numerical computation gives $b_4 < 0$, confirming that neither
  log-concavity nor $\infty$-log-concavity holds.
\end{example}

\begin{example}
  \label{ex:cosh}
  \textbf{($\mathcal{L}$-fixed sequences.)}
  Consider $a_n = \cosh(n)$.
  By Theorem~\ref{thm:fixed-char}, this sequence is fixed by
  $\mathcal{L}$ (taking $\theta = 1$).
  Since $\cosh(n) > 0$ for all $n$, Theorem~\ref{thm:fixed-point}
  gives that $(a_n)$ is $\infty$-log-concave.
  In contrast, $a_n = \cos(n\pi/3)$, which is also fixed by
  $\mathcal{L}$ (with $\theta = \pi/3$), changes sign and is
  therefore not log-concave.
\end{example}

\section{Open Problems}
\label{sec:open}

Our results raise several natural questions, which we state here.

\begin{enumerate}[label=\normalfont(\arabic*)]
  \item
    \textbf{(Cone of initial data for variable coefficients.)}
    Proposition~\ref{prop:initial-data} gives a matrix condition on
    $\mathbf{v}_0$ for log-concavity of a general P-recursive sequence.
    Characterize the geometry of the cone
    $\mathcal{C}_\infty = \{\mathbf{v}_0 : \mathcal{L}^i(a_n) \ge 0\;\forall\,n,i\}$
    for $d \ge 3$.
    Is it convex? Simply connected?

  \item
    \textbf{(Polynomial coefficients.)}
    Extend the quadratic form analysis to sequences where $c_k(n)$ is
    a polynomial of degree $r \ge 2$.
    In this case $Q_n$ is polynomial of degree $r$ in $n$, and the
    positive semi-definiteness analysis is more involved.

  \item
    \textbf{(Higher-order constant coefficients.)}
    For $d \ge 3$ with constant coefficients, characterize when
    $\mathcal{L}(a_n)$ is itself a sequence satisfying a
    constant-coefficient recurrence.
    If so, the tight criterion of Section~\ref{ssec:constant} would
    extend.

  \item
    \textbf{(Decidability.)}
    Given a P-recursive sequence specified by its recurrence and
    initial data, is $\infty$-log-concavity decidable in finite time?

  \item
    \textbf{(Connection to Boros--Moll.)}
    The rows of Pascal's triangle satisfy a second-order P-recursive
    relation.
    Apply the framework of Section~\ref{sec:logconcave} in that
    setting and relate it to the Boros--Moll
    conjecture~\cite{McNamaraSagan2010}.
\end{enumerate}

\bibliographystyle{elsarticle-num}

\begin{thebibliography}{99}

\bibitem{BirkhoffTrjitzinsky1933}
G.D.~Birkhoff, W.J.~Trjitzinsky,
Analytic theory of singular difference equations,
Acta Math.\ 60 (1933) 1--89.

\bibitem{BorosMoll2004}
G.~Boros, V.~Moll,
Irresistible Integrals: Symbolics, Analysis and Experiments in the
Evaluation of Integrals,
Cambridge University Press, Cambridge, 2004.

\bibitem{Branden2011}
P.~Br\"{a}nd\'{e}n,
Iterated sequences and the geometry of zeros,
J.\ Reine Angew.\ Math.\ 658 (2011) 115--131.

\bibitem{BrandenChasse2015}
P.~Br\"{a}nd\'{e}n, M.~Chasse,
Infinite log-concavity for polynomial P\'{o}lya frequency sequences,
Proc.\ Amer.\ Math.\ Soc.\ 143 (2015) 5225--5229.

\bibitem{Brenti1995}
F.~Brenti,
Combinatorics and total positivity,
J.\ Combin.\ Theory Ser.\ A 71 (1995) 175--218.

\bibitem{EdsonYayenie2009}
M.~Edson, O.~Yayenie,
A new generalization of Fibonacci sequence and extended Binet's formula,
Integers 9 (2009) 639--654.

\bibitem{Giacomelli2025}
P.~Giacomelli,
On log-concave operator acting on sequences and series,
arXiv:2503.15425v2 (2025).

\bibitem{HouZhang2021}
Q.-H.~Hou, T.~Zhang,
Asymptotic $r$-log-convexity and P-recursive sequences,
J.\ Symbolic Comput.\ 107 (2021) 22--37.

\bibitem{Huh2022}
J.~Huh,
Combinatorics and Hodge theory,
in: Proceedings of the International Congress of Mathematicians,
Vol.~1, EMS Press, 2022, pp.~212--261.

\bibitem{Karlin1968}
S.~Karlin,
Total Positivity, Vol.~I,
Stanford University Press, Stanford, CA, 1968.

\bibitem{KauersHolonomicFunctions}
M.~Kauers,
The holonomic toolkit,
in: C.~Schneider, J.~Bl\"{u}mlein (Eds.),
Computer Algebra in Quantum Field Theory,
Springer, Vienna, 2013, pp.~119--144.

\bibitem{LiuWang2007}
L.L.~Liu, Y.~Wang,
On the log-convexity of combinatorial sequences,
Adv.\ Appl.\ Math.\ 39 (2007) 453--476.

\bibitem{LogconcavePrecursive2021}
E.X.W.~Xia,
Asymptotic log-concavity and $2$-log-concavity for P-recursive sequences
via Birkhoff--Trjitzinsky asymptotic theory,
J.\ Difference Equ.\ Appl.\ 27 (2021) 1--20.

\bibitem{McNamaraSagan2010}
P.R.W.~McNamara, B.E.~Sagan,
Infinite log-concavity: developments and conjectures,
Adv.\ Appl.\ Math.\ 44 (2010) 1--15.

\bibitem{MedinaStratub2016}
L.A.~Medina, A.~Straub,
On multiple and infinite log-concavity,
Ann.\ Comb.\ 20 (2016) 125--138.

\bibitem{Saumard2014}
A.~Saumard, J.A.~Wellner,
Log-concavity and strong log-concavity: a review,
Stat.\ Surveys 8 (2014) 45--114.

\bibitem{Stanley1989}
R.P.~Stanley,
Log-concave and unimodal sequences in algebra, combinatorics and
geometry,
Ann.\ New York Acad.\ Sci.\ 576 (1989) 500--535.

\bibitem{Stanley1999}
R.P.~Stanley,
Enumerative Combinatorics, Vol.~2,
Cambridge Studies in Advanced Mathematics, vol.~62,
Cambridge University Press, Cambridge, 1999.

\bibitem{Wagner1992}
D.G.~Wagner,
Total positivity of Hadamard products,
J.\ Math.\ Anal.\ Appl.\ 163 (1992) 459--483.

\bibitem{WimpZeilberger1985}
J.~Wimp, D.~Zeilberger,
Resurrecting the asymptotics of linear recurrences,
J.\ Math.\ Anal.\ Appl.\ 111 (1985) 162--176.

\end{thebibliography}

\end{document}